\documentclass[onefignum,onetabnum]{siamart171218} 
\usepackage[margin=1in]{geometry}
\usepackage[english]{babel}
\usepackage[latin1]{inputenc}
\usepackage{amsfonts}
\usepackage{amssymb}
\usepackage{amsmath}
\usepackage{graphicx}
\usepackage{epsfig}
\usepackage{lipsum}
\usepackage{amsfonts}
\usepackage{graphicx}
\usepackage{epstopdf}
\usepackage{algorithmic}
\ifpdf
  \DeclareGraphicsExtensions{.eps,.pdf,.png,.jpg}
\else
  \DeclareGraphicsExtensions{.eps}
\fi

\newcommand{\R}{\mathbb{R}}

\newcommand{\PP}{\mathbb{P}}

\DeclareMathOperator{\tr}{Tr}

\newtheorem{prop}[theorem]{Proposition}
\newtheorem{defi}[theorem]{Definition}

\headers{Positivity for signed graph Laplacians}{I. Agbanusi, J.C. Bronski,
  and D. Kielty}
\title{A moment inequality and positivity for signed graph Laplacians \thanks{\textbf{Funding:} This work was supported by the National Science Foundation
  under grant NSF-DMS 1615418}}
\author{Ikemefuna Agbanusi\thanks{Colorado College, Department of
    Mathematics and Computer Science, Colorado Springs, CO 80903 
  (\email{iagbanusi@coloradocollege.edu}.)}
\and Jared C. Bronski\thanks{University of Illinois, Department of
  Mathematics, Urbana, IL 61801 
  (\email{bronski@illinois.edu}, \email{dkielty2@illinois.edu}.)}
\and Derek Kielty\footnotemark[3]}

\begin{document}

\maketitle

\begin{abstract} A number of recent papers have considered signed graph Laplacians, a generalization of the classical
graph Laplacian, where the edge weights are allowed to take either
sign. In the classical case, where the edge weights are all positive,
the Laplacian is positive semi-definite with the dimension of the
kernel representing the number of connected components of the
graph. In many applications one is interested in establishing
conditions which guarantee the positive semi-definiteness of the
matrix. In this paper we present an inequality on the eigenvalues of a
weighted graph Laplacian (where the weights need not have any
particular sign) in terms of the first two moments of the edge weights. This bound
involves the eigenvalues of the equally weighted Laplacian on the
graph as well as the eigenvalues of the adjacency matrix of the line
graph (the edge-to-vertex dual graph). For a regular graph the bound
can be expressed entirely in terms of the second eigenvalue of the
equally weighted Laplacian, an object that has been extensively
studied in connection with expander graphs and spectral measures of
graph connectivity. We present several examples including
Erd\H{o}s--R\'enyi random graphs in the critical and subcritical
regimes, random large $d$-regular graphs, and the complete graph, for
which the inequalities here are tight.

\end{abstract}


\section{Introduction}
There are a number of problems in applied mathematics where one is led
to consider the eigenvalues of a signed (combinatorial) graph
Laplacian: given a graph $G$ with $N$ vertices and $E$ edges the signed combinatorial
Laplacian is an $N\times N$ matrix of
the form 
\begin{equation}
L_{ij} (\boldsymbol{\gamma})= \begin{cases} 
  \sum_{k\neq i} \gamma_{ik} & i = j \\
  -\gamma_{ij} & i \neq j,~~ i \sim j\\
  0 & i \neq j,~~ i \not\sim j . 
\end{cases}
\label{eqn:MatLap}
\end{equation}
Here $i \sim j$ denotes the relation that distinct vertices $i$ and
$j$ are connected by an edge,  $\gamma_{ij}$ denotes the weight of edge $ij$ and $\boldsymbol{\gamma} \in {\mathbb R}^E$ is the vector of all edge weights. In this note Laplacian matrices are always symmetric, that is, $\gamma_{ij} = \gamma_{ji}$ for each $i$ and $j$. Note that the vector $\boldsymbol{1}_N=(1,1,1,\ldots,1)^t$ is always in
 the kernel of $L(\boldsymbol{\gamma})$. In the classical case where the edge weights are positive,
$\gamma_{ij}\geq 0$, the matrix is positive semi-definite but in
this paper the edge weights $\gamma_{ij}$ are not assumed to have any
particular sign. An incomplete list of the applications of such matrices includes:
\begin{itemize}
\item Data mining in social
  networks \cite{kunegis.et.al.2010,Hsieh.Chiang.Dhillon.2012}.
\item Hypergraph clustering algorithms \cite{Li}.
\item Models for the evolution of multi-agent
  networks \cite{Hu.Zheng.2013,Hu.Zheng.2014}.
\item Finding the fastest mixing linear consensus
  model \cite{Xiao.Boyd.2004,Xiao.Boyd.Diaconis.2004}.
\item The stability of  phase-locked solutions to the Kuramoto and
  related models \cite{Xu.Egerstedt.Droge.Schilling.2013,Jadbabaie.Motee.Barahona.2004,Dorfler.Bullo.2012,Mirollo.Strogatz.2005,Bronski.DeVille.Park.2012,DeVille.Ermentrout.2016}.
\end{itemize}

In many of these applications one is interested in establishing the
semi-definiteness of a Laplacian matrix, which typically implies
stability of the associated fixed point or a consensus state. For this
reason a number of papers have considered the problem of establishing
semi-definiteness of the Laplacian
matrix, as in \cite{Bronski.DeVille.2014,Zelazo.Burger.2014,Basar1.2016,Basar2.2016,Zelazo.Burger.2017}.

The purpose of this note is to present an inequality (Theorem \ref{thm:Main}) on the
eigenvalues of a Laplacian matrix in terms of the first two moments ---
the mean and variance ---  of
the edge weights. Following this we give a proof of the inequality and applications to both deterministic and random graphs. We first define:
\begin{equation}
Q = \frac{1}{E} \sum _{i>j} \gamma_{ij} \quad \text{and}  \quad P = \frac{1}{E} \sum _{i>j} \gamma_{ij}^2.
\label{eqn:PQ}
\end{equation}
Note that the Cauchy--Schwartz inequality implies that $P - Q^2 \geq
0$.  

In what follows we will denote the $i^{\text{th}}$ eigenvalue of a symmetric $N \times N$ matrix $L$ by $\lambda_i(L)$, numbered in increasing order of absolute values $|\lambda_1(L)| \leq |\lambda_2(L)| \leq \dots \leq |\lambda_N(L)|$. 
\begin{defi}
For a connected graph let $L(\boldsymbol{1}_E)$ denote the equally
weighted Laplacian---the
(combinatorial) Laplacian on the graph $G$ with all edge weights taken
to be unity: $\gamma_{ij}=1$.
\begin{equation*}
  L_{ij} (\boldsymbol{1}_E)= \begin{cases} \deg(v_i) & i = j\\
    -1 & i \neq j, ~ i \sim j \\
 0 & i \neq j,~ i \sim j,  \end{cases}
\end{equation*}
where $\deg(v_i)$ is the degree of vertex $i$. In particular, let $\lambda_i^G = \lambda_i(L(\boldsymbol{1}_E))$ denote the $i^{th}$ eigenvalue
of $L(\boldsymbol{1}_E)$, numbered in increasing order
\begin{equation}
\label{eqn:EigG}
0 = \lambda_1^G < \lambda_2^G \leq \lambda_3^G \leq \dots \leq \lambda_N^G.
\end{equation}
\end{defi}

To state our main result let $A^{LG(G)}$ denote the adjacency matrix of the line graph of $G$ and define the quantity
\[
  \mu = \max_{\boldsymbol{\gamma}\in {\mathbb{R}}^E  ~:~ \boldsymbol{\gamma} \perp \boldsymbol{1}_E } \frac{\langle {\boldsymbol{\gamma}}, (4 I + A^{LG(G)})
    \boldsymbol{\gamma}\rangle }{\Vert \boldsymbol{\gamma}\Vert^2}.
  \]
One has the inequalities
  \begin{equation}
  \label{eq:mu}
      4+ \lambda_{E-1}( A^{LG(G)}) \leq \mu \leq 4+ \lambda_{E}(
    A^{LG(G)}) \leq 2 d_{max} + 2, 
  \end{equation}
where $d_{max}$ is the maximum degree of the vertices in the graph $G$. In the case of a $d$-regular graph we have the equality
$\mu = 2 d + 2 - \lambda_2^G = 4 + \lambda_{E-1}(A^{LG(G)})$ (see the proof of Theorem \ref{thm:Main}).

\begin{theorem}
\label{thm:Main}
Consider a weighted Laplacian matrix $L(\boldsymbol{ \gamma})$ defined as in (\ref{eqn:MatLap}) on a
connected graph $G$ with $\lambda_2^G$ and $\lambda_N^G$ as in (\ref{eqn:EigG}) and $N\geq
3$. If $P$ and $Q$ are defined as in (\ref{eqn:PQ}) then the $N-1$ eigenvalues of
$L(\boldsymbol{\gamma})$ 
corresponding to eigenvectors orthogonal to  ${\bf1}_N$ satisfy the inequality  
\[
Q\lambda_2^G - \sqrt{E(P-Q^2) \mu \frac{N-2}{N-1}}    \leq
\lambda_i(L(\boldsymbol{\gamma})) \leq Q \lambda_N^G + \sqrt{E(P-Q^2) \mu \frac{N-2}{N-1}}. 
\]
Further this inequality is tight: for the complete graph there are
choices of edge weights realizing the upper and lower bounds.

In particular, if $G$ is a $d$-regular graph then
\[
Q\lambda_2^G - \sqrt{E(P-Q^2) (2d+2-\lambda_2^G)\frac{N-2}{N-1}}    \leq
\lambda_i\left(L(\boldsymbol{\gamma})\right) \leq Q \lambda_N^G + \sqrt{E(P-Q^2) (2d+2-\lambda_2^G)\frac{N-2}{N-1}}. 
\]
\end{theorem}

It is notable that, at least in the case of a regular graph, the
lower bound  depends only on $\lambda_2^G$, the second largest
eigenvalue of the graph Laplacian. The second eigenvalue is, of course, a
well-studied object that encodes important geometric information on
the connectivity of the graph, dating back at least to the work of
Feidler \cite{Feidler}, and is closely connected with the theory of
expander graphs. See, for instance, the review article of Hoory, Linial,
and Wigderson \cite{HLW.2006} for an overview of this area. 

The lower bound in Theorem \ref{thm:Main} is most important when considering the question of the
positivity of $L(\boldsymbol{\gamma})$. It shows that $L(\boldsymbol{\gamma})$
is positive definite if the variance $P - Q^2 \geq 0$ is sufficiently small compared to the mean squared. Specifically, $L(\boldsymbol{\gamma})$ is positive definite if
\begin{equation}
\label{eqn:Positivity}
\frac{(\lambda_2^G)^2}{\mu} E^{-1} Q^2 >  P - Q^2.
\end{equation}

Of course if the variance is small enough then each edge weight is necessarily positive, and hence the
Laplacian is necessarily positive semi-definite. A computation shows that all the edge weights are positive if
\begin{equation}
\label{eqn:Naive}
E^{-1} Q^2 > P - Q^2.
\end{equation}
Thus, when $(\lambda_2^G)^2/\mu >1$ inequality (\ref{eqn:Positivity}) gives an improvement on the naive condition (\ref{eqn:Naive}) in the sense that it allows for a larger variance. In Section \ref{sec:Examples} we present several examples where $(\lambda_2^G)^2/\mu >1$ for graphs with a large number of vertices. We also give an example of a graph where $(\lambda_2^G)^2/\mu \leq 1$. A simple example is the extreme case where $G$ is disconnected so that $\lambda_2^G=0$.

The inequalities in Theorem \ref{thm:Main} are also reminiscent of the Samuelson inequality for a finite set of real numbers. The original Samuelson inequality states that a finite set of real numbers is contained in a ball with center equal to the mean and radius proportional to the standard deviation of its elements (see \cite{JensenStyan} for example).

\begin{proof}[Proof of Theorem \ref{thm:Main}]
Recall that
${\bf 1}_E\in {\mathbb R}^{E}$ represents the vector ${\bf 1}_E
= (1,1,1,\ldots,1)^t$ so we have the orthogonal decomposition 
\[
\boldsymbol{\gamma} = Q {\bf 1}_E +  \tilde{\boldsymbol{\gamma}}, 
\]
where $\tilde{\boldsymbol{\gamma}}$ satisfies 
\begin{equation}
\label{eq:Constraint3}
\langle \tilde{\boldsymbol{\gamma}}, {\bf 1}_E\rangle = 0 \quad \text{and} \quad \Vert \tilde{\boldsymbol{\gamma}}\Vert^2 = E (P-Q^2), 
\end{equation}
and $\Vert\cdot \Vert$ is the Euclidean norm.

This gives a decomposition of the graph Laplacian into a ``mean'' and
``fluctuation'' as follows
\[L(\boldsymbol{\gamma}) = L(Q {\bf 1}_E + \tilde{\boldsymbol{\gamma}}) = Q L({\bf 1}_E) +  L(\tilde{\boldsymbol{\gamma}}).\]
Recall that for symmetric matrices $A$ and $B$ we have the inequalities
\[\lambda_{min} (A+B) \geq \lambda_{min}(A) + \lambda_{min}(B)  \quad \text{and} \quad
  \lambda_{max}(A+B) \leq \lambda_{max}(A) + \lambda_{max}(B),\]
where $\lambda_{min}(A) = \min_i \lambda_i(A)$ and $\lambda_{max}(A) = \max_i \lambda_i(A)$ are the smallest and largest eigenvalues of $A$. Thus, to prove the theorem it is enough to bound the spectral radius of the fluctuation $L(\tilde{\boldsymbol{\gamma}})$ and apply the above inequalities with $A = Q L({\bf 1}_E)$ and $B = L(\tilde{\boldsymbol{\gamma}})$.

To bound the fluctuation recall that the square of the Hilbert--Schmidt norm $\Vert \cdot \Vert_{HS}^2$ of a matrix is the sum of the squares of its eigenvalues, that is
\begin{equation}
  \Vert L(\tilde{\boldsymbol{\gamma}})\Vert_{HS}^2  = \sum_{i=1}^N \lambda_j^2.
  \label{eqn:Constraint1}
\end{equation}
Also note that
\begin{equation}
\label{eqn:Constraint2}
  \lambda_1 = 0 \quad \text{and} \quad \tr(L(\tilde{\boldsymbol{\gamma}})) = 2 \sum_{i>j}  \tilde \gamma_{ij} =
    \sum_{i=1}^N \lambda_i =0. 
\end{equation}
Maximizing $|\lambda_i|$ subject to the constraints (\ref{eqn:Constraint1}) and (\ref{eqn:Constraint2}) we have
\[
  \max_i |\lambda_i| \leq \sqrt{\frac{N-2}{N-1}} \Vert L(\tilde{\boldsymbol{\gamma}})\Vert_{HS}.
 \]

Next we express the Hilbert--Schmidt norm as a quadratic form in the
 edge-weights $\tilde \gamma_{ij}:$
\begin{equation}
\Vert L(\tilde{\boldsymbol{\gamma}})\Vert_{HS}^2 = \sum_i \lambda_i\left(L(\tilde{\boldsymbol{\gamma}})\right)^2 = 2 \sum_{i<j} \tilde \gamma_{ij}^2 +  \sum_i \bigg( \sum_{j \neq i} \tilde \gamma_{ij}\bigg)^2. 
\label{eqn:QuadForm}
\end{equation}
To prove the estimate, we
consider the Hilbert--Schmidt norm as a
quadratic form on $\boldsymbol{\gamma} \in {\mathbb R}^{E}$ and maximize it subject
to the constraints (\ref{eq:Constraint3})

The quadratic form on $\mathbb{R}^E$ in
(\ref{eqn:QuadForm}) can be expressed in graph-theoretic terms as $\langle \boldsymbol{ \gamma}, (4 I_{E \times E
} + A^{LG(G)}) \boldsymbol{\gamma}\rangle,$ where $A^{LG(G)}$ is the adjacency
matrix of the line graph of the graph $G$. The line
graph $LG(G)$ has a vertex set given by the edge set of the original
graph $G$. Two vertices in the line graph are adjacent if the
corresponding edges in $G$ share a vertex. Thus we have that
\[
  \Vert L(\tilde{\boldsymbol{\gamma}}) \Vert_{HS}^2 \leq E(P-Q^2) \max_{\tilde{\boldsymbol{\gamma}}\in {\mathbb{R}}^E  ~:~ \tilde{\boldsymbol{\gamma}} \perp {\bf 1}_E }
  \frac{\langle \tilde{\boldsymbol{\gamma}}, (4 I + A^{LG(G)}) \tilde{\boldsymbol{\gamma}}\rangle}{\Vert
    \tilde{\boldsymbol{\gamma}}\Vert^2},
\]
from which it follows that
\begin{equation}
\label{eq:muBnd}
 \max_i |\lambda_{i}(L(\tilde{\boldsymbol{\gamma}})| \leq \sqrt{E(P-Q^2)\frac{N-2}{N-1}\mu}.
\end{equation}

When $G$ is $d$-regular we will prove that $\mu = 4 + \lambda_{E-1}(A^{LG(G)}) = 2d + 2- \lambda_2^G$ by relating the eigenvalues of the $A^{LG(G)}$ to the eigenvalues of $L(\boldsymbol{1}_E)$ --- the Laplacian of the original graph. 
In this case, the line graph $LG(G)$ is
also $d$-regular, the vector ${\bf 1}_E$ is the eigenvector
corresponding to the largest eigenvalue, and thus the maximum of the Rayleigh quotient in (\ref{eq:muBnd}) is equal to the second largest eigenvalue of $4 I_{E \times E} + A^{LG(G)}$. It is well-known, and easy
to see, that the adjacency matrix of the line graph is related to the
(unoriented) incidence matrix $C$ of the graph $G$ by
\[
  A^{LG(G)} = C^T C - 2 I_{E\times E},
\]
and that the adjacency matrix $A^{G}$ of the original graph is related to the
(unoriented) incidence matrix by
\[
  A^{G} = C C^T - D,
\]
where $D$ is the degree matrix --- the diagonal $N \times N$ matrix with
the vertex degrees along the diagonal. Since $G$ is $d$-regular we have that $D = d I_{N \times N}$ so that
\begin{align*}
  A^{LG(G)} + 4 I_{E\times E} &= C^T C + 2 I_{E \times E}\\
  C C^T &= 2d I_{N \times N} - L(\boldsymbol{1}_E). 
\end{align*}
To conclude the proof recall that the non-zero eigenvalues of $C^T C$
are equal (counting by multiplicity) to the non-zero eigenvalues of $C
C^T$. Thus, the second-largest eigenvalue of $A^{LG(G)}+4I_{E\times E}$, and therefore $\mu$,
is equal to $2d + 2 -\lambda_2^G$, where $\lambda_2^G$ is the
second-smallest eigenvalue of graph Laplacian.
\end{proof}

\section{Examples}
\label{sec:Examples}

In this section we present examples of Theorem \ref{thm:Main} applied to the complete graph
on $N$ vertices, the Erd\H{o}s--R\'enyi random graph in the critical and supercritical scaling
regime, the cyclic graph, and random $d$-regular graphs. Recall that $L(\boldsymbol{\gamma})$ is positive semi-definite whenever $P$ and $Q$ satisfy (\ref{eqn:Positivity}) or (\ref{eqn:Naive}). In each example except the cyclic graph, the quantity $(\lambda_2^G)^2/\mu > 1$ so that condition (\ref{eqn:Positivity}) is an improvement over the naive condition (\ref{eqn:Naive}). 

\subsection{Complete graph}
Note that the complete graph is ``universal''; since any graph on $N$
vertices is a subgraph of $K_N$ the complete graph inequality applies to
any graph, although one can expect to do better with an inequality
that includes information about the topology of the graph in
question. The example of the complete graph is also interesting in
that the upper and lower bounds are actually attained --- while it is
clear that each inequality in the derivation of 
Theorem \ref{thm:Main}
is tight it is not immediately clear that there is a single example
for which all of the inequalities are extremized.

For the complete graph the mean, $Q L({\bf 1}_E)$ is a constant multiple
of the orthogonal projection onto the $N-1$ dimensional subspace
$(1,1,1,\ldots,1)^{\perp}$. It is easy to see that the eigenvalues of $Q L({\bf 1}_E)$ are
given by $0$, with multiplicity $1$, and $N Q$, with multiplicity
$N-1$. It is also noteworthy that in the case where the underlying
topology is the complete graph the mean  $Q L({\bf 1}_E)$ commutes with
{\em every} combinatorial Laplacian, and thus with the fluctuation
$L(\tilde{\boldsymbol{\gamma}})$.

The line graph of the complete graph $K_N$ is the Johnson
$J_{N,2}.$ Using the fact that $K_N$ is regular of degree $N-1$ or
known results about the spectrum of the adjacency matrix of the
Johnson graph it follows that
\[Q N - \sqrt{E (P-Q^2) (2(N-1)+2- N) \frac{N-2}{N-1}} \leq
  \lambda_i(L(\boldsymbol{\gamma})) \leq Q N + \sqrt{E (P-Q^2) (2(N-1)+2- N)
    \frac{N-2}{N-1}},\] 
    or equivalently, 
\[N \bigg( Q - \sqrt{ \frac{N-2}{2}}\sqrt{P-Q^2} \bigg) \leq
  \lambda_i(L(\boldsymbol{\gamma})) \leq N \bigg( Q + \sqrt{
    \frac{N-2}{2}}\sqrt{P-Q^2} \bigg).\]

This is the sharp version of a simple inequality for
the complete graph case that was proven in
\cite{Agbanusi.Bronski.2018} via the Hilbert--Schmidt equality in order to establish
the existence of a spectral gap. In particular it was shown there
that if $L(\boldsymbol{\gamma})$ is a graph Laplacian 
with weights given by $\boldsymbol{\gamma} \in \R^E$, then 

\begin{equation}
\label{rough}
N \bigg( Q - \sqrt{N-1} \sqrt{P - Q^2} \bigg)\leq \lambda_i(L(\boldsymbol{\gamma})) \leq N \bigg( Q + \sqrt{N-1} \sqrt{P - Q^2} \bigg)
\end{equation}
so the current inequality improves on the elementary estimate by roughly a
factor of $\sqrt{2}$ for large $N$. Furthermore an example in the paper of
Agbanusi and Bronski \cite{Agbanusi.Bronski.2018} shows that the current inequality is sharp: there exist
explicit Laplace matrices for which the upper and lower limits are
achieved.

\subsection{Erd\H{o}s--R\'enyi critical scaling}
Consider an Erd\H{o}s--R\'enyi random graph in the critical
regime, where the edge probability is $p = \frac{p_0\log N}{N}$ with
$p_0>1$ to ensure connectivity of the graph. It has been shown by Kolokolnikov, Osting, and Von Brecht \cite{KOV}
that in the critical scaling regime one has that
\[
  \lambda_2 \sim a(p_0) p_0 \log N + O(\sqrt{\log N}), \quad \text{as } N \to \infty,
\]
where $a(p_0) \in (0,1)$ is defined to be the solution to $p_0 -1 = a p_0 (1-\log(a)).$ The inequality holds in the sense that 
\[
  \bigg\vert \frac{\lambda_2}{N p} - a(p_0) \bigg\vert \leq \frac{C_1}{N p} 
\]
is true with probability at least $1-C_2 \exp\{- C_3 \sqrt{N p}\}$ for
some constants $C_1,C_2,C_3$.

We are not aware of any precise results for $\mu$, but it is fairly
easy to get (probabilistic) upper bounds since (\ref{eq:mu}) says that $\mu \leq 2 d_{max} + 2$. It follows a union
bound argument (see Appendix \ref{sec:App}) that there is a constant $C_4 > 0$ such that
\[{\mathbb P}\left(\max_i \deg(v_i) \leq C p_0 \log N\right) \geq 1 - C_4 N^{\beta(C)}
\]
where
\[
  \beta(C) = 2- p_0 - C p_0 \log C + C p_0.  
  \]
We can choose any $C$
such that $\beta(C)<0$. For simplicity if we take $C = 4$ we have
that $\beta \approx 1-2.55p_0.$ 

Since this is a random graph we also need an estimate of $E$, the total number of edges. Since the edges are
independent this essentially follows from the central limit theorem,
and we have that $E = \frac{p_0(N-1)}{2} \log N + o(N^{1/2 + \epsilon})$ for each $\epsilon > 0$ with high
probability. Combining the above we have that the {\em nonzero} eigenvalues satisfy
the lower bound
\[
  \lambda_i \geq p_0 \log(N) \bigg(a(p_0) Q \big( 1+o(1) \big) - \sqrt{4 (N-2) \big( 1+o(1) \big)} \sqrt{P-Q^2} \bigg)
  \]
with probability tending to $1$ as $N \rightarrow \infty$.

The upper
bound follows similarly --- we are not aware of any result on the
precise distribution of the largest eigenvalue of the Laplacian of an
Erd\H{o}s--R\'enyi graph, but the largest eigenvalue is obviously less
than twice the largest degree of the graph, giving
\[
  \lambda_i \leq p_0 \log N \bigg( 8Q \big( 1+o(1) \big) + \sqrt{4(N-2) \big( 1+o(1) \big)} \sqrt{P-Q^2} \bigg).
  \]

 Thus for an Erd\H{o}s--R\'enyi graph in the critical scaling regime
 the Laplacian is (with probability tending to $1$) positive definite
 if the inequality
 \[
   Q^2 > \frac{4(N-1)}{a^2(p_0)} (P-Q^2).
 \]
Note that with high probability the number of edges will be $p_0 N
\log N $ so the above estimate is asymptotically better than the naive
estimate (\ref{eqn:Naive}) by a factor of $\log N$. The constant in the above
is obviously not sharp, as we have used a crude estimate on the
largest eigenvalue of the adjacency matrix, and moreover no use has
been made of the constraint that $\tilde{\boldsymbol{\gamma}}$ is mean
zero. We do, however, expect that the scaling with $N$ is tight. 
 
\subsection{Erd\H{o}s--R\'enyi supercritical scaling} Now we consider the Erd\H{o}s--R\'enyi graphs in the supercritical regime with fixed edge probability $p \in (0,1)$. Observe that since $p \geq p_0 \log(N)/N$ for large $N$ we have that the graph is connected almost surely. Moreover, in this regime the average degree of a vertex is $pN$. In fact, a similar calculation to the one in Appendix \ref{sec:App} shows that
\[\mathbb{P}\bigg( \max_i \deg(v_i) \leq (1 + N^{-1/2 + \epsilon})pN \bigg) \to 1, \quad \text{as} \quad N \to \infty,\]
for $\epsilon \in (0,1/2)$.

In particular, this shows that
\[\mu \leq 2(1 + N^{-1/2 + \epsilon})pN + 2 \quad \text{and} \quad \lambda_N^G \leq 2(1 + N^{-1/2 + \epsilon})pN,\]
with probability tending to 1 as $N \to \infty$. By Theorem 2 in \cite{Juhasz}, for each $\epsilon > 0$ we have
\[\lambda^G_2 = pN + o(N^{\frac{1}{2} + \epsilon}), \quad \text{as} \quad N \to \infty.\]

The number of edges is $E = pN(N-1)/2 + o(N^{1 + \epsilon})$ as $N \to \infty$. Applying these bounds in the non-regular case of Theorem \ref{thm:Main} we have
\[pN \bigg(Q (1 + o(1)) - \sqrt{(N-2)\big( 1 + o(1) \big)} \sqrt{P-Q^2} \bigg) \leq \lambda_i \leq pN \bigg(2Q\big( 1 + o(1) \big) + \sqrt{(N-2) \big( 1 + o(1) \big)} \sqrt{P-Q^2} \bigg)\]
with probability tending to 1 as $N \to \infty$, implying positivity
when
\[
Q \gtrsim N (P-Q^2).
\]
Notice that when we
take $p = 1$ we recover the non-sharp bounds for the complete graph
topology in (\ref{rough}) for large $N$. This is again due to the fact
that we do not employ the constraint that $\tilde{\boldsymbol{\gamma}}$ has mean zero. The same comments that were made for the critical case apply
here as well --- the constants can be improved but we believe the
scaling to be optimal.

\subsection{Cyclic graph}

For the Cyclic graph on $N$ vertices, the graph Laplacian with all edge weights equal to 1 is twice the identity plus the circulant matrix generated by the vector $c = (0,-1,0,\dots,0,-1)$. The eigenvalues of the circulant matrix, and hence those of the Laplacian, can be computed explicitly. The smallest nonzero and largest eigenvalues are $\lambda_2^G = 2(1 - \cos(2 \pi/N))$ and $\lambda_N^G = 2$, respectively. Since the Cyclic graph has degree 2 we have the bounds $\mu =6 - \lambda_2^G$. Putting all this together we find that
\[2\bigg(Q(1 - \cos(2 \pi/N)) - \sqrt{3}\sqrt{\frac{N}{2}\frac{N-2}{N-1}} \sqrt{P - Q^2} \bigg) \leq \lambda_i \leq 2 \bigg( Q + \sqrt{3}\sqrt{\frac{N}{2} \frac{N-2}{N-1}}\sqrt{P - Q^2} \bigg).\]

In this example the naive inequality (\ref{eqn:Naive}) on $P$ and $Q$ is actually the stronger one since $(\lambda_2^G)^2/\mu \leq 2^2(1 - \cos(2 \pi /N))^2/2 < 1$ for all large enough $N$. This is in contrast to each of the other examples where $(\lambda_2^G)^2/\mu > 1$ for large $N$. 

\subsection{Random \texorpdfstring{$d$}{Lg}-regular graphs}
Consider the probability space consisting of $d$-regular graphs $(d \geq 3)$ on $N$ vertices with the uniform probability measure. Work of Freidman \cite{Freidman} implies that in this setting
\[\lambda_2^G \geq d - 2\sqrt{d-1} + o(1), \quad \text{as}~ N \to \infty,\]
with high probability. Applying the above inequality and that $\lambda_N^G \leq 2d$ to the bound for $d$-regular graphs in Theorem \ref{thm:Main} we have that

\[\lambda_i \geq d \bigg( Q \big( 1 - 2d^{-1/2} + o(1) \big) - \sqrt{\frac{N}{2} \big( 1 + 4d^{-1/2} + o(1) \big)}\sqrt{P-Q^2} \bigg),\]
and
\[\lambda_i \leq d \bigg(2Q + \sqrt{\frac{N}{2} \big(1 + 4d^{-1/2} + o(1) \big)} \sqrt{P-Q^2}\bigg). 
\]

Since $\mu = 2d + 2 - \lambda_2^G$ for $d$-regular graphs we have
\begin{equation}
\label{eq:dReg}
\frac{(\lambda_2^G)^2}{\mu} \geq \frac{d^2 - 4d \sqrt{d-1} + 4(d-1) + o(1)}{d + 2\sqrt{d-1} + 2 + o(1)}, \quad \text{as}~ N \to \infty,
\end{equation}
with high probability. Thus, for large $d$ the right side of (\ref{eq:dReg}) is roughly of size $d$, and in particular, $(\lambda_2^G)^2/\mu > 1$ almost surely as $N \to \infty$.

\section{Concluding Remarks}
In this paper we have derived bounds on the largest and smallest
eigenvalues of a graph Laplacian in terms of the mean and variance of
the edge weights and the second largest eigenvalue of the equally
weighted graph Laplacian. These inequalities are tight in the case of
the complete graph topology --- there exist edge weightings which
attain both the upper and the lower bounds. 

There are a couple of ways in which it might be interesting to extend
these results. Firstly while the bounds are tight for the complete
graph it is unlikely that this is the case for most graph
topologies. In the course of the proof we use the inequality
\[
  \lambda_{min}(A+B) \geq \lambda_{min}(A) + \lambda_{min}(B),  
  \]
where $A$ is the equally weighted Laplacian and $B$ is the fluctuation. In
the case of the complete graph topology the equally weighted Laplacian
is the identity on mean zero vectors, $A$ and $B$ commute, and this
inequality is actually an equality. This is not true for other
underlying graph topologies. It would be interesting to explore the
extent to which this inequality fails to be tight for topologies other
than the complete graph topology.

A second question concerns the quantity $\mu$, which is related to the maximum of a Rayleigh quotient for the adjacency of the line graph over mean zero vectors in
${\mathbb R}^E$. In the regular case we can, via a duality argument,
compute $\mu$ in terms of the second largest eigenvalue of the graph
Laplacian. For the non-regular case, we only bound $\mu$ in terms of  the maximum degree, which does not exploit
the mean zero condition at all. It would be interesting to develop a bound on $\mu$ in the non-regular case that exploits
the mean zero condition.

  {\bf Acknowledgements:} J.C.B. would like to acknowledge support
  from the National Science Foundation under grant NSF-DMS 1615418. D.K. would like to acknowledge support from the University of Illinois Campus Research Board award RB19045 (to Richard Laugesen) and the U.S. Department of Education through the Graduate Assistance in Areas of National Need (GAANN) program.

\appendix
\section{Upper bound on the maximum degree}
\label{sec:App}

\begin{prop}
Suppose that $G$ is an Erd\H{o}s--R\'enyi random graph where each
possible edge is present with probability $p = \frac{p_0 \log N}{N}$
with $p_0>1$. The probability that
\[\max_i \deg(v_i) \leq C p_0 \log(N)\]
tends to $1$ algebraically as $N \to \infty$ for $C$ large enough. Choosing $C \geq 4$ is sufficient.
\end{prop}

\begin{proof}
Observe that since 
\[\mathbb{P}\bigg(\max_i \deg(v_i) \leq C p_0 \log(N) \bigg) = 1 - \mathbb{P} \bigg( \deg(v_i) > C p_0 \log(N) ~ \text{for some} ~ i \bigg),\] 
by a union bound it suffices to show that
\[N\PP\bigg(\deg(v_i) > C p_0 \log(N) \bigg) \to 0, \quad \text{as} \quad N \to \infty.\]

Recall that $\deg(v_i)$ follows a binomial distribution so that for $K \in \mathbb{N}$ we have
\[\PP(\deg(v_i) > K) = \sum_{j = K}^{N-1} {N-1 \choose j} p^j (1-p)^{N-1 - j} \leq  \sum_{j = K}^N {N \choose j} p^j (1-p)^{N - j}.\]
The inequality follows since the probability that an event occurs in $N$ trials is larger than the probability that it occurs in $N-1$ trials. Now suppose that $K \geq K_{max}$ --- the $K$ for which the maximum of $K \mapsto {N \choose K} p^K (1-p)^{N - K}$ is achieved. In this case 
\[\PP(\deg(v_i) > K) \leq N {N \choose K} p^K (1-p)^{N - K}.\]

Now we apply Sterlings approximation to estimate $N \choose K$ to find that
\[\PP(\deg(v_i) > K) \lesssim N \frac{N^N e^{-N}\sqrt{2 \pi N}}{K^K e^{-K}\sqrt{2 \pi K}(N - K)^{N-K}e^{-(N-K)}\sqrt{2 \pi (N-K)}} p^K (1-p)^{N-K}.\]
After regrouping and some elementary estimates we have 
\[\PP(\deg(v_i) > K) \lesssim N \sqrt{\frac{N}{K(N-K)}} \bigg( \frac{Np}{K} \bigg)^K \bigg( \frac{N(1-p)}{N-K} \bigg)^{N-K}.\]
Now we choose $K = CNp$ for $C > 1$ so that $K \geq K_{max}$ for large $N$ since $K_{max} \leq \lfloor (N+1)p \rfloor$. It follows that 

\[\PP(\deg(v_i) > K) \lesssim N \bigg( \frac{1}{C} \bigg)^{C p_0 \log(N)} \bigg( \frac{1 - \frac{p_0 \log(N)}{N}}{1 - \frac{C p_0 \log(N)}{N}}\bigg)^{N-K} \lesssim N \bigg( \frac{1}{C} \bigg)^{C p_0 \log(N)} \bigg( \frac{1 - \frac{p_0 \log(N)}{N}}{1 - \frac{C p_0 \log(N)}{N}}\bigg)^N\]
Using that $(1 - p_0\log(N)/N)^N = e^{N \log(1 - p_0\log(N)/N)}$ and Taylor expanding the outer logarithm we have
\[\PP(\deg(v_i) > K) \lesssim N \bigg( \frac{1}{C} \bigg)^{C p_0 \log(N)} \frac{e^{-p_0 \log(N)}}{e^{-C p_0 \log(N)}} = N N^{-C p_0 \log(C)} N^{-p_0} N^{C p_0}.\]
Combining the exponents shows that
\[N\PP\bigg(\deg(v_i) > C p_0 \log(N) \bigg) \lesssim N^{2-C \log(C) p_0 + C p_0 - p_0}.\]
Since $p_0 > 1$ choosing $C > 3.6$ gives algebraic decay. In particular, choosing $C = 4$ gives 
\[\PP\bigg(\max_i\deg(v_i) > 4 p_0 \log(N)\bigg) \lesssim N^{1 - 2.54 p_0}.\]

\end{proof}

\bibliographystyle{unsrt}
\bibliography{SharpInequality}

\end{document}